\documentclass[12pt,twoside]{amsart} 
\usepackage{amssymb,color,tikz}
\usetikzlibrary{matrix} \nonstopmode \textwidth=16.00cm

\usepackage[utf8]{inputenc}
\usepackage[T1]{fontenc}
\usepackage{tikz-cd}
\usepackage{lipsum}
\usepackage{adjustbox}
\usepackage{graphicx,color}
\usepackage{comment}
\usepackage{xcolor}
\usepackage{centernot}
\usepackage{mathtools}
\usepackage{stmaryrd}
\usepackage{blkarray}
\usepackage{amsmath}
\usepackage{nicematrix}
\usepackage{enumitem}
\usepackage{yhmath}
\usepackage[all]{xy}

\numberwithin{equation}{section} 

\newcommand {\PP}{\mathbb{P}}

\newcommand{\K}{\mathbb{K}}

 \DeclareMathOperator{\Proj}{Proj}

\def\cocoa{{\hbox{\rm C\kern-.13em
      o\kern-.07em C\kern-.13em o\kern-.15em A}}}

\DeclareMathOperator{\HF}{HF}

\def\cocoa{{\hbox{\rm C\kern-.13em
      o\kern-.07em C\kern-.13em o\kern-.15em A}}}

\DeclareMathOperator{\ann}{Ann}
\DeclareMathOperator{\Hom}{Hom}
\DeclareMathOperator{\Tor}{Tor}

\DeclareMathOperator{\Ext}{Ext}

\newtheorem{theorem}{Theorem}[section]
\newtheorem{corollary}[theorem]{Corollary}

\newtheorem{proposition}[theorem]{Proposition}

\theoremstyle{definition}
\newtheorem{definition}[theorem]{Definition}
\newtheorem{remark}[theorem]{Remark}
\newtheorem{Notation}[theorem]{Notation}
\newtheorem{example}[theorem]{Example}
\newtheorem{notation}[theorem]{Notation}

\usepackage{hyperref}
\definecolor{MyDarkGreen}{cmyk}{0.7,0,1,0}

\newcommand{\h}[1]{\-\mbox{-#1}}

\title[AG algebras with Macaulay dual generator with fixed Waring rank]{Artinian Gorenstein algebras with Macaulay dual generator with fixed Waring rank}

\author[J. Martins]{Janaine Martins}
    \address{ ICEx - UFMG, Departmento de Matemática, Av. Antônio Carlos, 6627, 30123-970 Belo Horizonte, MG, Brasil} \email{mesquitajanaine@gmail.com, ORCID 0000-0002-4350-3770}

\author[R.\ M.\ Mir\'o-Roig]{Rosa M.\ Mir\'o-Roig (*)}
\address{Facultat de
Matem\`atiques i Inform\`atica, Universitat de Barcelona, Gran Via de les
Corts Catalanes 585, 08007 Barcelona, Spain} \email{miro@ub.edu,  ORCID 0000-0003-1375-6547}

\thanks{(*) Corresponding author}



\thanks{The first author was supported by CAPES/PDSE grand number 88881.220483/2025-01. The second author was partially supported
by the grant PID2024-157142NB-I00.}

\begin{document}

\begin{abstract}
 In this paper, we prove that the Artinian Gorenstein $\K$-algebra $A_{F_s}$ of codimension $n$, socle degree $d$ and  Macaulay dual generator $F_s := \ell_1^d + \dots + \ell_s^d \in \mathbb{K}[X_1, \dots, X_n]$ where $\ell _1, \cdots , \ell_s$ are general linear forms satisfies the strong Lefschetz property (SLP). This result allows us to study whether the Waring rank of $F_s$ is exactly $s$.
 
 Furthermore, we show that $A_{F_s}$ is the doubling of a suitable 0-dimensional scheme $Z_{F_s}$ in $\mathbb{P}^{n-1}$, the so-called tight annihilating scheme of  $A_{F_s}$, and we compute the minimal free resolution of $A_{F_s}$ in terms of the minimal free $R$-resolution of $I(Z_{F_s})$. Finally, we determine the linear general Jordan type of $A_{F_s}$.
\end{abstract}

\maketitle

\section{Introduction}

The study of graded Artinian Gorenstein (AG for short)  $\K$-algebras $A$ holds a significant place in commutative algebra due to their ubiquity in many fields of mathematics, including algebraic geometry, commutative algebra, algebraic topology and combinatorics; and due to  their rich structure and numerous applications \cite{MW09}. Over the last decades, many different problems of such $\K$-algebras have been extensively studied among which we single out a few: 

\vskip 2mm
(1) To determine whether or not they have  the weak Lefschetz property (WLP - see Definition \ref{def_wlp}) or the Strong Lefschetz property (SLP - see Definition \ref{def_wlp}). If they do, it has strong implications for the Hilbert function;
for example, not only must these functions be unimodal but in fact they must satisfy a
strong growth condition. 

(2) To determine whether or not they are the doubling of a 0-dimensional scheme, the so-called tight annihilating scheme of $A$ (see Definition \ref{doubling}). If they do, it has several important geometric implications. 

(3) To compute their minimal free resolution or at least their graded Betti numbers; these homological invariants completely determine the graded Poincaré series of $A$. 

(4) To find the generic Jordan type of $A$, i.e., the Jordan blocks of the nilpotent endomorphism $\times \ell: A\longrightarrow A$ where $\ell $ is a generic form.

\vskip 2mm

Central to the study of standard graded AG $\K$-algebras $A$ is
the concept of a Macaulay dual generator, which, via Macaulay-Matlis duality, connects every finitely
generated Artinian Gorenstein graded $\K$-algebra $A$ to a single homogeneous polynomial $F$, the Macaulay dual generator of $A$ (see Theorem \ref{Gor}). This
duality  transforms algebraic and homological  properties
of $A$ into questions about the polynomial $F$, streamlining their exploration.

In this work, we focus on a specific class of AG algebras, denoted by $A_{F_s}$,  of codimension $n$, socle degree $d$, and with Macaulay dual generator $F_s$  a sum of $d$-th powers of $s$ general linear forms 
$$F_s := \ell_1^d + \dots + \ell_s^d \in \mathbb{K}[X_1, \ldots, X_n].$$
Notice that the Waring rank of $F_s$ $(Wr(F_s))$ is $\le s$ (see Definition \ref{def_wr}).

Our goal will be to address the above 4 problems for such $\K$-algebras $A_{F_s}$. We will first determine their $h$-vector and using induction on the socle degree we will  prove that they do satisfy WLP. This result will allow us to analyze whether the inequality $Wr(F_s)\le s$ turns out to be an equality.

  Since $F_s$ is a sum of $d$-th powers of $s$ general linear forms, the study of $A_{F_s}$ is intrinsically related to the geometry of a 0-dimensional scheme $Z_{F_s} = \{P_1, \dots, P_s\}$ of degree $s$ in the projective space $\mathbb{P}^{n-1}$ where each point $P_i$ corresponds to the dual of the linear form $\ell_i$. Indeed, $A_{F_s}$ is a doubling of $Z_{F_s}$ (see Definition \ref{doubling}). Therefore, the graded Betti numbers of $A_{F_s}$ are determined by the graded Betti numbers of  $I(Z_{F_s})$ and we are led to consider the following longstanding question: What is the structure of the minimal free resolution (MFR) of $I(Z_{F_s})$?
The structure of the MFR for a general set of points in $\PP^2$ has been known since the work of Gaeta \cite{G1}, and was later extended to $\mathbb{P}^3$ by Ballico and Geramita \cite{BG}. For the general case in $\mathbb{P}^n$, the minimal resolution conjecture (MRC for short) was formally proposed by Lorenzini \cite{L}. This conjecture states that the resolution is numerically determined solely by the Hilbert function of the points. The MRC posits that the Betti table of the coordinate ring $R/I(X)$ of a general set of points $X\subset \PP^n$ is as sparse as possible, concentrated in at most two consecutive rows and satisfying the complementarity condition
$$\beta_{i, j} \cdot \beta_{i+1, j+1} = 0$$
for all $i, j$, being $\beta_{i,j} := \text{Tor}_i^R[R/I(X), \mathbb{K}]_{j}$. While this conjecture was proven for either small dimensions ($n\le 4$) or for a large number of points  \cite{HS}, work by Eisenbud, Popescu, Schreyer and Walter \cite{EPSW} has shown that it fails in $\mathbb{P}^n$, with $n \geq 6$ and $n \neq 9$.

Despite the known counterexamples to the MRC in higher dimensions, the conjecture remains a fundamental benchmark for the expected behavior of the syzygies for general sets of points. In our work, we utilize the doubling construction combined with the mapping cone process to derive the minimal free resolution of the AG algebra $A_{F_s}$ from the known (or expected) resolution of the 0-dimensional scheme $X$. Finally, we provide a deeper study of $A_{F_s}$ describing the Jordan decomposition
of the multiplication map $\times \ell :A_{F_s}\longrightarrow A_{F_s}$ by a linear form $\ell$.
The Jordan type $P_{\ell,A}$ of a graded Artinian algebra
$A$ and a linear form $\ell \in A_1$ is the partition of $\dim_{\K} A$ determining the
Jordan block decomposition for the multiplication map $\times \ell :A\longrightarrow A$. From the generic Jordan type $P_{\ell,A_{F_s}}$ of $A_{F_s}$ we deduce that $A_{F_s}$ has not only the WLP but also the SLP.

\vskip 4mm

We now present a detailed overview of the main contents and the structure of this paper. In Section \ref{section2}, we recall the definitions of the Hilbert function and the Lefschetz properties for Artinian graded  $\K$-algebras, review Macaulay-Matlis duality as a very powerful tool for constructing AG $\K$-algebras, and collect the basic properties about the Jordan type of a graded Artinian $\K$-algebra. Section \ref{section3} is devoted to determining the $h$-vector of $A_{F_s}$ and to analyzing the Waring rank $Wr(F_s)$ of $F_s$, noting that it is at most $s$, and  proving that it equals $s$ for  $s \leq \binom{\lfloor d/2 \rfloor +n-1}{n-1}$. In the same Section, we also show that $A_{F_s}$ has the WLP. In Section \ref{section4}, we revisit the doubling and mapping cone constructions in order to describe the shape of a minimal free resolution  of $A_{F_s}$. Finally, in Section \ref{section5}, we determine the generic linear Jordan type of $A_{F_s}$, and as application we conclude that $A_{F_s}$  has the SLP.

\vskip 4mm
\noindent {\bf Acknowledgement.}   The
authors thank the Referees for accurate reading and useful comments.
The first author thanks the Universitat de Barcelona for its hospitality.

\section{Notation and background}\label{section2}

Throughout this paper $\mathbb{K}$ will denote an algebraically closed field of characteristic zero,  $R$ will denote the polynomial ring $\mathbb{K}[x_1,\dots,x_n]$ while $S$ denotes the polynomial ring $\mathbb{K}[X_1, \ldots, X_n]$.

Given any standard graded \mbox{$\K$-algebra} \({A=R/I}\), where $I$ is a homogeneous ideal of $R$, we denote by $\xymatrix@1{\HF_A:\mathbb{Z} \ar[r] & \mathbb{N}}$, with \({\text{HF}_A(j)=\dim _{\mathbb{K}}[A]_j}\),
its \emph{Hilbert function}.

If $A$ is Artinian, there exists an integer $c$, called \text{socle degree}, such that $A_c \ne 0$ and  $A_i = 0$ for all $i > c$. In this case, the finite sequence
\begin{eqnarray*}
(h_0, h_1, \dots, h_c)
\end{eqnarray*}
where $h_i:=\HF_A(i)>0$, is called the \textit{h-vector} of $A$. The \textit{Sperner number} $S_A$ of a standard graded Artinian $\K$-algebra $A$, with $h$-vector $h_A=(h_0,\dots,h_c)$, is the maximum value of $h_A$, i.e., $S_A=\max_{i\in\{0,\dots,c\}}\{h_i\}$.

\subsection{Macaulay-Matlis duality} In this subsection, we quickly recall the construction of the standard graded Artinian Gorenstein (AG, for short) $\K$-algebra
associated to a given form.

To set up the framework for this duality, let $V$ be an $n$-dimensional $\mathbb{K}$-vector space, and let $\{x_1, \dots, x_n\}$ and $\{X_1, \dots, X_n\}$ be dual bases of $V^{\ast}$ and $V$, respectively. This allows us to naturally identify the previously defined polynomial rings $R$ and $S$ with the symmetric algebras
\begin{eqnarray*}
R=\bigoplus_{i \geq 0} \operatorname{Sym}^i V^{\ast}
\quad \text{and} \quad
S=\bigoplus_{i \geq 0} \operatorname{Sym}^i V .
\end{eqnarray*}

Under this identification, $R$ acts on $S$ via differential operators defined by 
\begin{eqnarray*}
x_i = \frac{\partial}{\partial X_i}.
\end{eqnarray*}

This action, called the \emph{apolar action} of $R$ on $S$, identifies elements in $R$ with partial differential operators on $S$. More precisely,  the apolar action of $R$ on $S$ induces bilinear maps defined by differentiation 
\begin{eqnarray*}
\operatorname{Sym}^i V ^{\ast}\otimes \operatorname{Sym}^j V
&\longrightarrow& 
\operatorname{Sym}^{j-i} V\\
\alpha \otimes F &\longmapsto& \alpha(F),
\end{eqnarray*}
which is zero whenever $i>j$. Given a graded $R$-submodule $M$ of $S$, we define its annihilator in $R$ by
\begin{eqnarray*}
\operatorname{Ann}_R(M)
=
\left\{
\alpha = p(x_1,\dots,x_n) \in R
\;\middle|\;
\alpha(F)
=
p\!\left(\frac{\partial}{\partial X_1}, \dots, \frac{\partial}{\partial X_n}\right) \circ  F = 0 \text{ for all } F \in M
\right\}
\end{eqnarray*}
which is a homogeneous  ideal of $R$. Conversely, if $I \subset R$ is a homogeneous ideal, we define its \emph{inverse system} by
\begin{eqnarray*}
I^{-1}
=
\{F \in S \mid \alpha(F)=0 \text{ for all } \alpha \in I\}.
\end{eqnarray*}
Then $I^{-1}$ is a graded $R$-submodule of $S$.

This establishes a one-to-one correspondence (Macaulay--Matlis duality):
\begin{eqnarray*}
\{\text { homogeneous ideals of } R\} & \leftrightarrow &  \{\text{graded } \text { $R$-submodules of } S\} \\
\operatorname{Ann}_R(M) & \leftarrow & M \\
I & \rightarrow & I^{-1}
\end{eqnarray*}

Under this correspondence, $I^{-1}$ is a finitely generated $S$-module 
if and only if $R/I$ is an Artinian algebra.
\begin{definition} Let $A$ be a standard Artinian  graded $\K$-algebra.
    We say that $A$ is \textit{Gorenstein} with socle degree $c$ if $\dim_{\mathbb{K}} [A]_c = 1$
and, for every $i=0,\ldots,c$, the multiplication map induces a perfect pairing
$$
[A]_i \times [A]_{c-i} \longrightarrow [A]_c \simeq \mathbb{K}.
$$
Equivalently, for each $i$ there is a natural isomorphism
$$
[A]_{c-i} \simeq \operatorname{Hom}_{\mathbb{K}}([A]_i,\mathbb{K}).
$$
\end{definition}
\begin{theorem}\label{Gor}
Let 
\begin{eqnarray*}
A=\bigoplus_{i=0}^c [A]_i = R/I
\end{eqnarray*}
be a standard graded Artinian $\mathbb{K}$-algebra. 
Then $A$ is Gorenstein if and only if there exists a homogeneous polynomial 
$F \in [S]_c$ such that
\begin{eqnarray*}
A \simeq R/\operatorname{Ann}_R(F).
\end{eqnarray*}
\end{theorem}
\begin{proof}
See (\cite{MW09}, Theorem 2.1).
\end{proof}

Notice that if $F \in [S]_c$ and $A := R/\operatorname{Ann}_R(F)$, we consider the apolar map defined by:
\begin{eqnarray*}
\varphi_F : [R]_t & \longrightarrow & [S]_{c-t} \\
\alpha & \longmapsto & \alpha(F)
\end{eqnarray*}
By the First Isomorphism Theorem, we have the isomorphism of $\mathbb{K}$-vector spaces:
$$ [R]_t / \ker(\varphi_F) \simeq \operatorname{Im}(\varphi_F) $$
Since $\ker(\varphi_F) = [\operatorname{Ann}_R(F)]_t = [I]_t$, this induces an isomorphism:
$$ [A]_t \simeq \left\langle \frac{\partial ^{t}F}{\partial X_1^{t_1}\cdots \partial X_n^{t_n}} \;\middle|\; t_1+\cdots +t_n= t \right\rangle $$
(see, for instance, \cite[Section 1.1] {IK} for further details).

Consequently, the Hilbert function of $A$ can be computed as
\begin{equation}
    \label{computingHilbertFunctionGorenstein}
\begin{array}{lcl}
    \HF_A(t) &= &\dim_{\K}[A]_t \\
    &= &\dim_{\K}[A]_{c-t} \\
    & =&\dim_{\K}[I^{-1}]_{c-t}  \\
    &= &\dim_{\K} \left\langle \frac{\partial ^{t}F}{\partial X_1^{t_1}\cdots \partial X_n^{t_n}} \mid t_1+\cdots +t_n= t
    \right\rangle .
\end{array}
\end{equation}

\subsection{Lefschetz properties}

In this subsection, we define the Lefschetz properties, which are inspired by the 
Hard Lefschetz Theorem in topology \cite{S80}.

\begin{definition}\label{def_wlp}
Let $A=R/I$ be a standard graded Artinian $\K$-algebra with $h$-vector $(h_0,\dots,h_c)$. We say that $A$ has the \emph{weak Lefschetz property} (WLP, for short)
if there is a linear form $\ell \in [A]_1$ such that, for all
integers $i\ge0$, the multiplication map
\[
 \times \ell: [A]_{i} \longrightarrow [A]_{i+1} 
\]
has maximal rank, i.e., it is injective or surjective. 
In this case, the linear form $\ell$ is called a \emph{weak Lefschetz element} of $A$. 

$A$ has the \emph{strong Lefschetz property} (SLP, for short) if there is a linear form $\ell \in [A]_1$ such that, for all
integers $i\ge0$ and $k\ge 1$, the multiplication map
\[
 \times \ell^k: [A]_{i} \longrightarrow [A]_{i+k} 
\]
has maximal rank. Such an element $\ell$ is called a \emph{strong Lefschetz element} of $A$.

\end{definition}

\begin{remark} \label{rmk}  \noindent (1)  We consider the short exact sequence
    \begin{equation*}
    [R/I]_{i}\xrightarrow{\times\ell}[R/I]_{i+1}\rightarrow[R/(I,\ell)]_{i+1}\rightarrow 0,    
    \end{equation*}    
    where $\ell\in [A]_1.$
    
    Then, the multiplication map $\times\ell$ fails to have maximal rank if and only if
    \[
    \dim_\K[R/(I,\ell)]_{i+1} > \max\{0,\dim_\K[R/I]_{i+1}-\dim_\K[R/I]_{i}\}.
    \]

More precisely:
\begin{itemize}
    \item [(i)] If $\dim_\K[R/I]_{i}\leq \dim_\K[R/I]_{i+1}$ and $\dim_\K[R/(I,\ell)]_{i+1} > \dim_\K[R/I]_{i+1}-\dim_\K[R/I]_{i}$, we say that $\ell$ is not a weak Lefschetz element because of failure of injectivity;
    \item [(ii)] If $\dim_\K[R/I]_{i}\ge \dim_\K[R/I]_{i+1}$ and $\dim_\K[R/(I,\ell)]_{i+1} > 0$, we say that $\ell$ is not a weak Lefschetz element because of failure of surjectivity.
\end{itemize}
\end{remark}

\noindent \label{AG}(2) In the case of an AG $\K$-algebra \(A\) of socle degree \(c\), the surjectivity of the map 
\[
\times \ell : [A]_i \to [A]_{i+1}
\]
implies the injectivity of 
\[
\times \ell : [A]_{c-i-1} \to [A]_{c-i},
\]
because of duality. Hence, for AG $\K$-algebras it will be enough to check if
the map $\times \ell: [A]_i \longrightarrow [A]_{i+1}$ is surjectivity in the middle, that is, $i=\lfloor c / 2\rfloor$  \cite[Proposition 2.1]{MMN2}.

    \begin{example} (1) Consider $I = (x^2, y^2, z^2, xy) \subset R=\mathbb{K}[x, y, z]$. The ideal $I$ has the  WLP. Since $R/I$ has $h$-vector (1, 3, 2), it suffices to check that the map  $\times (x+y+z): [R/I]_1 \rightarrow [R/I]_2$ is surjective or, equivalently,  $[R/(I,\ell)]_2=0.$
In fact,
\begin{eqnarray*}
[R/(I,\ell)]_2
&\cong&
\left[\mathbb{K}[x,y,z]/(x^2,y^2,z^2,xy,x+y+z)\right]_2 \\
&\cong&
\left[\mathbb{K}[x,y]/(x^2,y^2,(x+y)^2,xy)\right]_2\\
&\cong&
[\mathbb{K}[x,y]/(x^2,y^2,xy)]_2 = 0.
\end{eqnarray*}
(2) Let $I = (x^3, y^3, z^3, xyz) \subset R=\mathbb{K}[x, y, z]$. The ideal $I$ fails the  WLP. The $h$-vector of $R/I$ is  (1, 3, 6, 6, 3). 
For any linear form $\ell = ax + by + cz$, there exists a nontrivial element in the kernel of $\times \ell:  [R/I]_2 \longrightarrow [R/I]_3$, namely
\[
f = a^2 x^2 + b^2 y^2 + c^2 z^2 - ab\, xy - ac\, xz - bc\, yz.
\]
Therefore, $A$ fails the WLP.
\end{example}
    The second example is known as the Togliatti example (see \cite{BK} and  \cite{MMRMO}).


\subsection{Jordan type of a graded Artinian algebra}


\begin{definition} \label{jordanBasisDefinition} 
Let $A$ be a standard  graded Artinian $\K$-algebra. The \textit{Jordan type} of a linear form $\ell\in [A]_1$ in $A$ is the partition  of  $\dim_\K A$, $P_{\ell,A}=(p_1,\dots,p_s)$, $p_1 \ge \cdots \ge p_s$, where $p_1, \dots,p_s$ represent the block sizes in the Jordan  form of the map $\times \ell: A \longrightarrow A$. A $\K$-basis $\mathcal{B}$ of $A$ is called  a \textit{pre\h{Jordan}} basis if it is given by  $\mathcal{B}=\{\ell^iz_k:1\le k\le s, \ 0\le i \le p_k-1\}$ for suitable elements $z_k\in A$. A \textit{string} is a sequence
$$
S_k:=(z_k,\ell z_k,\dots, \ell^{p_k-1}z_k)
$$
of $\mathcal{B}$. A \textit{bead} is an element $\ell^iz_k$ in a string $S_k$. A \textit{Jordan basis} for $\ell$ is a pre\h{Jordan} basis $\mathcal{B}$ satisfying $\ell^{p_k}z_k=0$ for each $k\in\{1,\dots,s\}$. By  \cite[Lemma 2.2 (iii)]{IMM22}, one can construct a Jordan basis from any pre\h{Jordan} basis . 
If $\nu_k$ is the degree of $z_k$, then the sequence
\[
\mathcal{S}_{\ell,A}=((p_1,\nu_1),\dots,(p_s,\nu_s))
\]
is an invariant of $(A,\ell)$  called \textit{Jordan degree\h{type}} of $A$ with respect to $\ell$ \cite[Lemma 2.2(iv)]{IMM22}.
\end{definition}

 We illustrate the definition with an example.

 \begin{example}
     Let $A = \mathbb{K}[x,y,z]/(x^2, y^2, z^2, xyz)$. A basis for $A$ is, for instance,
$$ \{1, x, y, z, xy, xz, yz\}. $$
So, the $h-$vector is (1, 3, 3) and $\dim_\mathbb{K} A = 7$. Consider the multiplication map
$$ \times \ell : A \to A, $$
where $\ell = x + y + z$. We have the following Jordan basis strings:

\begin{center}
    $1 \longmapsto x+y+z \longmapsto 2(xy+xz+yz) \longmapsto 0$ \\
    \smallskip
    $x\longmapsto xy + xz \longmapsto 0$ \\
    \smallskip
    $y \longmapsto xy + yz \longmapsto 0$
\end{center}

  Thus, a Jordan basis $\mathcal{B}$ is composed by the following strings
    \[
    (1,\ell,\ell^2), (x,x\ell), \text{ and } (y,y\ell).
    \]
    So the Jordan form  of  $\times\ell:A\to A$ with respect to $\mathcal{B}$ is
    \[
    \newcommand{\?}[1]{\multicolumn{1}{c|}{#1}}
    \left(\begin{array}{ccccccc}
    0 & 0 & \?0 & 0 &   0 & 0 & 0 \\
    1 & 0 & \?0 & 0 &   0 & 0 & 0 \\
    0 & 1 & \?0 & 0 &   0 & 0 & 0 \\
    \cline{1-5}
    0 & 0 & \?0 & 0 & \?0 & 0 & 0 \\
    0 & 0 & \?0 & 1 & \?0 & 0 & 0 \\
    \cline{4-7}
    0 & 0 & 0 & 0 & \?0 & 0 & 0 \\
    0 & 0 & 0 & 0 & \?0 & 1 & 0
    \end{array}\right).
    \]    
Therefore, the Jordan type of $\ell$ on $A$ is 
$$
P_{\ell,A} = (3,2,2),
$$
since the strings have lengths $3$, $2$, and $2$, respectively.

Moreover, the strings start at degrees $0,1,1$, so the Jordan degree-type is
$$
\mathcal{S}_{\ell,A} = ((3,0),(2,1),(2,1)).
$$
 \end{example}

 \begin{notation}
    Whenever a length is repeated in a Jordan type, we will use exponentiation in order to abbreviate notation. For instance, if $A$ is the $\K$- algebra of the above  example, we will write $P_{\ell,A}=(3,2^2)$ and 
$\mathcal{S}_{\ell,A}=(3_0,2_1^2)$. The lower index indicates the degree where the string starts.
\end{notation}

Given a graded Artinian $\K$-algebra $A$, we say that a linear form $\ell\in [A]_1$ is \textit{general} if it belongs to a Zariski open dense subset of $[A]_1$.
The \textit{generic linear Jordan type} of $A$, denoted $P_{A}$, is the Jordan type $P_{\ell,A}$ for a general linear form $\ell $ (cf.\ \cite[Lemma 2.54, Definition 2.55]{IMM22}).

It is important to notice that computing all possible linear Jordan types for any Artinian $\K$-algebra is not feasible. We will devote the last Section of this paper to compute the generic linear Jordan type of $A_{F_s}$.

Let $\ell_i = a_{i,1}X_1 + \dots + a_{i,n}X_n \in [S]_1$ for $i = 1, \dots, s$. We say that the sequence of linear forms $\ell_1, \dots, \ell_s$ is \emph{general} (or in \emph{general position}) if the vector of coefficients $(a_{1,1}, \dots, a_{1,n}, \dots, a_{s,1}, \dots, a_{s,n}) \in \mathbb{K}^{ns}$ belongs to a suitable dense Zariski open subset of $\mathbb{K}^{ns}$.


\section{The Hilbert function and the weak Lefschetz property}\label{section3}

We fix  integers $n\ge 2$, $d \ge 1$ and 
$1 \le s \le \binom{d+n-1}{d}$. Let $F_{s} \in [S]_d$ be a homogeneous polynomial of degree $d$ given by the sum of the $d$-th powers of $s$ general linear forms, namely,
\begin{eqnarray*}
F_{s} = \ell_1^d + \cdots + \ell_s^d,
\end{eqnarray*}
where $\ell_1, \dots, \ell_s$ are general linear forms in $[S]_1$.

We are interested in studying the $\K$-algebra
\begin{eqnarray*}
A_{F_{s}} := R / \operatorname{Ann}(F_{s}),
\end{eqnarray*}
and its algebraic properties. Let us start computing its $h$-vector.

\begin{proposition}\label{h_vector}We fix integers $d$, $s\ge 1$ and $n\ge 2$ such that $s\le {n+d-1\choose n-1}$. Let $A_{F_{s}} = R/\ann(F_{s}),$ be the AG $\K$-algebra  of codimension $n$, socle degree $d$ and Macaulay dual generator 
\begin{eqnarray*}
F_{s}=\ell_1^d+\cdots+\ell_s^d
\end{eqnarray*}
with $\ell_1,\dots,\ell_s$ general linear forms in $[S]_1$.  
Then the Hilbert function of $A_{F_{s}}$ is given by
$$
h_k=\dim_{\mathbb K}[A_{F_{s}}]_k
=\begin{cases}\min\left\{\binom{n+k-1}{k},\, s\right\} \text{ for }  0\leq k\leq\lfloor\frac{d}{2}\rfloor \\
\text{symmetry}  \text{ for } \lfloor\frac{d}{2}\rfloor < k \leq d  .
\end{cases}
$$
\end{proposition}
\begin{proof}
  By the Equation (\ref{computingHilbertFunctionGorenstein}), the $k$-th homogeneous component $[A_{F_s}]_k$ of  $A_{F_s}$ can be  identified with the vector subspace of $S_{d-k}$
generated by the  $k$-th  partial derivatives of $F_s$. On the other hand, for a linear general form
\begin{eqnarray*}
\ell_i=a_{i}^1X_1+\cdots+a_{i}^nX_n,
\end{eqnarray*}
we obtain:
\begin{eqnarray*}
\frac{\partial^k}{\partial x_{j_1}\cdots\partial x_{j_k}}(\ell_i^d)
=
\frac{d!}{(d-k)!}
a_{i}^{j_1}\cdots a_{i}^{j_k}\ell_i^{d-k},
\end{eqnarray*}
and, we get:

$$
\begin{array}{rcl}
\frac{\partial^kF_s}{\partial x_{j_1}\cdots\partial x_{j_k}} &
= & 
\frac{\partial^k}{\partial x_{j_1}\cdots\partial x_{j_k}}(\sum _{i=1}^s\ell_i^d )\\  \\
& = &
\frac{d!}{(d-k)!}\sum _{i=1}^s
a_{i}^{j_1}\cdots a_{i}^{j_k}\ell_i^{d-k}.
\end{array}
$$

Moreover, we have
$$[A_{F_s}]_k=\left \langle  \frac{\partial^kF_s}{\partial x_{j_1}\cdots\partial x_{j_k}} \ \mid  \ 1\le j_1\le \cdots \le j_k\le n   \right\rangle .$$
Since $\left[A_{F_s}\right]_k=\operatorname{Span}\left\{\ell_1^{d-k}, \ldots, \ell_s^{d-k}\right\},$ its dimension is bounded above by both $\operatorname{dim}_{\mathbb{K}} R_k=\binom{n+k-1}{k}$ and by the number of generators $s$. Hence, $h_k \leq \min \left\{\binom{n+k-1}{k}, s\right\}.$

Now suppose first that
$s \leq\binom{ n+k-1}{k}$. As the linear forms $\ell_1, \ldots, \ell_s$ are general, the powers $\ell_1^{d-k}, \ldots, \ell_s^{d-k}$
are linearly independent in $S_{d-k}$. Therefore, $h_k=s.$

On the other hand, if $s > \binom{n+k-1}{k},$ then the only possible linear relations among the $k$-th derivatives are those coming from the space of differential operators itself. Since the linear forms are general, no additional relations occur, and consequently $h_k=\operatorname{dim}_{\mathbb{K}} R_k=\binom{n+k-1}{k}.$

Thus, for any $k$ such that $0\leq k\leq\lfloor\frac{d}{2}\rfloor$, it holds that
$$h_k=\dim_{\mathbb K}[A_{F_{s}}]_k
=\min\left\{\binom{n+k-1}{k},\, s\right\} $$
and, since $A_{F_s}$ is an AG $\K$-algebra  of socle degree $d$, the Hilbert function of $A_{F_s}$ is symmetric (i.e. $h_k = h_{d-k}$) and this finishes the proof. 
\end{proof}

\begin{example}
We  consider the AG $\K$-algebra $A_{F_s}$ of codimension 5 and socle degree 7 with Macaulay dual generator $F_s=\ell_1^7+\cdots +\ell_s^7\in \K [X_1,X_2,X_3,X_4,X_5]$ where $\ell_1, \ldots ,\ell_s$ are general linear forms. Applying Proposition \ref{h_vector} we get:

\vskip 2mm
\begin{itemize}
    
\item[(i)] If $F_1=X_1^7$ then the $h$-vector of  $A_{F_1}$ is  (1,1,1,1,1,1,1,1);

\item[(ii)] If $F_5=X_1^7+\cdots +X_5^7$ then the $h$-vector of  $A_{F_5}$ is  (1,5,5,5,5,5,5,1);

\item[(iii)] If $F_7=\ell _1^7+\cdots +\ell_7^7$ then the $h$-vector of  $A_{F_7}$ is  (1,5,7,7,7,7,5,1);

\item[(iv)]  If $F_{16}=\ell _1^7+\cdots +\ell_{16}^7$ then the $h$-vector of  $A_{F_{16}}$ is  (1,5,15,16,16,15,5,1); and

\item[(v)]  If $F_s=\ell _1^7+\cdots +\ell_s^7$, $s\ge 35$ then the $h$-vector of  $A_{F_s}$ is  (1,5,15,35,35,15,5,1).
\end{itemize}
\end{example}

\begin{definition} \label{def_wr}
    Let $0 \ne  f \in \K[x_1,\ldots, x_n]$ be a homogeneous polynomial of degree $d$.
The {\em Waring rank} of $f$, which will be denoted by $Wr(f)$, is defined as the minimum number
of terms in a expression of $f$ as a linear combination of powers of linear forms:
$$f =\sum _{i=1}^{Wr(f)} \ell _i^d
$$
where the $\ell_i \in \K[x_1,\ldots, x_n]$ are linear forms.
\end{definition}

\begin{corollary} Fix integers $d,s\ge 1$ and $n\ge 2$. Consider a homogeneous form $F_s\in \K[X_1,\ldots ,X_n]$ of degree $d$. Assume $s\le {\lfloor d/2 \rfloor+n-1\choose n-1}$ and  $F_s=\ell_1^d+\cdots +\ell_s^d$ where $\ell_1, \ldots ,\ell_s$ are general linear forms in $\K[X_1,\ldots ,X_n]$. Then,  
    $Wr(F_s)=s$.
\end{corollary}
\begin{proof}  By definition,  we  have $Wr(F_{s}) \le s$ and, by Proposition \ref{h_vector}, we know the Sperner number of $A_{F_s}=R/\ann(F_s)$, i.e., $S_{A_{F_s}}=s$.  Suppose that $W r\left(F_s\right)=r$. Then $F_s=L_1^d+\cdots+L_r^d$ for some linear forms $L_1, \ldots, L_r$. By Equation (\ref{computingHilbertFunctionGorenstein}), the homogeneous component $\left[A_{F_s}\right]_k$ is generated by the $k$-th partial derivatives of $F_s$. Since each $k$-th partial derivative is a linear combination of the forms $L_1^{d-k}, \ldots, L_r^{d-k}$ it follows that $\operatorname{dim}_{\mathbb{K}}\left[A_{F_s}\right]_k \leq r$ for every $k.$ Therefore, $S_{A_{F_s}} \leq r=W r\left(F_s\right).$ Since $S_{A_{F_s}}=s$, we obtain $s \leq W r\left(F_s\right).$ Together with the inequality $\operatorname{Wr}\left(F_s\right) \leq s$, this yields
$W r\left(F_s\right) = s$.
\end{proof}

We will now prove the main result of this Section. Let us start with a technical result that will play an important role in our proof.

\begin{proposition}
\label{exactsequence}
Let $R/I$ be an AG $\K$-algebra and set $I=\operatorname{Ann}(F)$. Then, for every linear form $\ell\in [R/I]_1$ the sequence
\begin{equation}
\label{seq}
0\longrightarrow \frac{R}{(I:\ell)}(-1)\xrightarrow[]{\times\ell} \frac{R}{I}\longrightarrow \frac{R}{(I,\ell)}\longrightarrow 0
\end{equation}
is exact. Moreover $\frac{R}{(I:\ell)}$ is an AG $\K$-algebra with $\ell\circ F$ as dual generator.
\end{proposition}
\begin{proof}
We get the result cutting the exact sequence
\[
0\longrightarrow \frac{(I:\ell)}{I}(-1)
\longrightarrow \frac{R}{I}(-1)\xrightarrow{\,\,\,\times\ell\,\,\,}\frac{R}{I}\longrightarrow \frac{R}{(I,\ell)}\longrightarrow 0
\]
into two short exact sequences. As for the second fact, notice that 
$$
(\text{Ann}(F) : \ell)=\text{Ann}(\ell\circ F).
$$
Indeed,
\begin{align*}
    f\in \text{Ann}(\ell\circ F) &\iff f\circ(\ell\circ F)=0 \\
                           &\iff (f\ell)\circ F=0 \\
                           &\iff f\ell \in \text{Ann}(F) \\
                           &\iff f\in (\text{Ann}(F) : \ell).
\end{align*}
Now the result follows from Theorem \ref{Gor}.
\end{proof}

\begin{theorem} \label{WLP_main}
We consider the AG $\K$-algebra $A_{F_{s}} = R/\operatorname{Ann}(F_{s}),$ of codimension $n$, socle degree $d$ and Macaulay dual generator 
\begin{eqnarray*}
F_{s}=\ell_1^d+\cdots+\ell_s^d
\end{eqnarray*}
where $\ell_1,\ldots,\ell_s$ are general linear forms in $[S]_1$, and $1 \le s \le \binom{d+n-1}{d}$.                 
Then, $A_{F_s}$ has the WLP. 
\end{theorem}
\begin{proof}
   We argue by induction on the socle degree $d$ of $A_{F_{s}}$.
   
    For $d = 1$, the $h$-vector of  $A_{F_s}$ is  (1, 1). Therefore, the multiplication by any nonzero linear form has maximal rank, and $A_{F_s}$ has the WLP.
    
Assume the result holds for any AG $\K$-algebra of codimension $n$ and socle degree $d-1$ whose dual generator is $ F_{s,d-1}=L_1^{d-1}+\cdots+L_s^{d-1}$ where 
$L_1,\dots,L_s$ are general linear forms in $[S]_1$.  We will now prove that $A_{F_s}$ has the WLP. To this end we consider the linear form $\ell =x_1+\cdots +x_n$ and the exact sequence  (\ref{seq})
\begin{equation}
\label{seq1}
0\longrightarrow \frac{R}{\ann(\ell \circ F_s)}(-1)\xrightarrow[]{\times\ell} \frac{R}{\ann (F_s)}\longrightarrow \frac{R}{(\ann(F_s),\ell)}\longrightarrow 0
\end{equation}
given in Proposition \ref{exactsequence}. 

A straightforward computation shows that $\ell \circ F_s=L_1^{d-1}+\cdots+L_s^{d-1}$ where 
$L_1,\dots,L_s$ are general linear forms in $[S]_1$. Therefore, the $h$-vectors of $\frac{R}{\ann(\ell \circ F_s)}$ and $\frac{R}{\ann (F_s)}$ are given by Proposition \ref{h_vector}. This  allows us to deduce that $\left[\frac{R}{(\ann(F_s),\ell)}\right]_t=0$ for all $t\ge \left\lfloor \frac{d}{2} \right\rfloor+1$ or, equivalently,

$$
\times \ell : [A_{F_s}]_{t-1} \longrightarrow [A_{F_s}]_t,
$$
is surjective for all  $t\ge \left\lfloor \frac{d}{2} \right\rfloor+1$. Finally, applying Remark \ref{rmk}(2) we deduce that $A_{F_s}$ has the WLP.
\end{proof}


\section{The minimal free resolution of $A_{F_s}$}\label{section4}

 Let $M$ be a finitely generated  $R$-module. It is well known that it has a minimal graded free $R$-resolution of the following type:
$$
0\longrightarrow F_{n+1} 
\longrightarrow F_n \longrightarrow \cdots  \longrightarrow F_i \longrightarrow \cdots \longrightarrow F_1\longrightarrow  F_0\longrightarrow M\longrightarrow 0 
$$
where
$$ F_i=\bigoplus _jR(-j)^{\beta_{ij}^R(M)}
$$
and the graded Betti numbers $\beta_{ij}^R(M)$ of $M$ over $R$ are defined as usual as  the integers
$$
\beta_{i,j}:=\beta_{ij}^R(M)=\dim_\K [\Tor^R_i(M,\K)]_j.
$$ 

These homological invariants are one of  our main focus and indeed our goal in this Section is to determine the graded Betti numbers $\beta _{ij}^R(A_{F_s})$ of $A_{F_s}$.

 A compact way to display the graded Betti numbers $\beta _{i,j}$  of a finitely generated $R$-module is the {\em Betti} table, that is,

    \vskip 4mm
    
    \begin{center}
    \begin{tabular}{ r | c c c c}
           &  0 &  1 &  $\cdots$ & $n+1$ \\
    \hline
     $\vdots$ &  $\vdots$ &  $\vdots$ & $\ddots$ & $\vdots$ \\
       $k$ &  $\beta_{0, k}$ & $\beta_{1, k+1}$ & $\cdots$ &  $\beta_{n + 1, k +n+1}$ \\ 
      $k+1$ &  $\beta_{0, k+1}$ & $\beta_{1, k+2}$ & $\cdots$ & $\beta_{n + 1,k+n+2}$ \\
 $\vdots$ &  $\vdots$ &  $\vdots$ & $\ddots$ & $\vdots$ \\
    \end{tabular}
    \end{center}

We will start with a toy example that will illustrate our general approach.
\begin{definition}\label{canonicalmodule}
    The {\em canonical module}  of a graded $R$-module $M$ is defined:
$$
\omega_R:=\Ext^{n+1-\dim M}_R(M,R(-n-1)).$$
    
\end{definition}
\begin{example}
Let $A_{F_5}$ be the AG $\K$-algebra of codimension 3 and socle degree 7 with Macaulay dual generator  $$F_5=X^7+Y^7+Z^7+(X+Y+Z)^7+(X-Y+2Z)^7.$$

Note that these five linear forms are general in the sense that no three of them share a common zero set in $\mathbb{P}^2$, equivalently, their corresponding dual points in $\mathbb{P}^2$ are in linear general position (no three are collinear).

By Proposition  \ref{h_vector}  it has $h$-vector $(1, 3, 5, 5, 5,5, 3,1)$ and  straightforward computation shows that 
\begin{eqnarray*}
    \ann_R(F_5) =  &&(4xy-xz-3yz,4y^2z-3xz^2-yz^2,4x^2z-3xz^2-yz^2,\\
    &&128y^6-
      xz^5-127yz^5-64z^6,128x^6-387xz^5+3yz^5+192z^6)
\end{eqnarray*}

We now consider the 0-dimensional subscheme $P\subset \PP^2=\Proj(\K[x,y,z])$, $P=\{(1:0:0),(0:1:0),(0:0:1),(1:1:1),(1:-1:2)\}$ with homogeneous ideal $I(P)=(4xy-xz-3yz, 4y^2z-3xz^2-yz^2, 4x^2z-3xz^2-yz^2)$.  The minimal free $R$-resolution of $R/I(P)$ is:
$$0 \longrightarrow R(-4)^2 \xlongrightarrow{\varphi_2} \begin{array}{c} R(-2) \\ \oplus \\ R(-3)^2 \end{array} \xlongrightarrow{\varphi_1} R \longrightarrow R/I(P) \longrightarrow 0$$
where
$$\begin{aligned}
\varphi_1 &=  \begin{pmatrix} xy - \frac{1}{4}xz - \frac{3}{4}yz & y^2z - \frac{3}{4}xz^2 - \frac{1}{4}yz^2 & x^2z - \frac{3}{4}xz^2 - \frac{1}{4}yz^2 \end{pmatrix}\\
\varphi_2 &= \begin{pmatrix} -xz & -yz \\ \frac{1}{4}z & x - \frac{3}{4}z \\ y - \frac{1}{4}z & \frac{3}{4}z \end{pmatrix}.
\end{aligned}$$
        From Definition  \ref{canonicalmodule} we have $\omega_{R/I(P)} \cong \operatorname{Ext}^2_R(R/I(P), R(-3))$. Therefore, dualizing the resolution of $R/I(P)$ and twisting by $-3$ yields the minimal free resolution of $\omega_{R/I(P)}$. To match the socle degree $c = 7$ of $A_{F_5}$, we must further twist this complex by $-7$. 
        
       Dualizing and twisting by $-10$ we get the minimal free $R$-resolution of $\omega_{R/I(P)}(-7)$. That is:
    $$
    0 \longrightarrow R(-10) \longrightarrow \begin{array}{c} R(-8) \\ \oplus \\ R(-7)^2 \end{array} \longrightarrow R(-6)^2 \longrightarrow \omega_{R/I(P)}(-7) \longrightarrow 0.
    $$

The lowest degree generators of $\omega_{R/I(P)}(-7)$ appear in degree 6. If we send the two generators of $\omega_{R/I(P)}(-7)$, which have degree $6$, to the generators $128y^6-
      xz^5-127yz^5-64z^6$ and $128x^6-387xz^5+3yz^5+192z^6$ of the ideal $\ann_R(F_5)$, we obtain the following short exact sequence:
    \begin{equation}\label{aux_seq}     
 0 \longrightarrow \omega_{R/I(P)}(-7) \xlongrightarrow{\psi} R/I(P) \longrightarrow A_{F_5}=R/\ann_R(F_5)\longrightarrow 0.
       \end{equation}

       Finally, using the exact sequence (\ref{aux_seq}) and applying the mapping cone process (the details of this construction are presented in the discussion following Definition \ref{doubling}). We get the following minimal free resolution of $A_{F_5}$:

           $$
    0 \longrightarrow R(-10) \longrightarrow \begin{array}{c} R(-8) \\ \oplus \\ R(-7)^2 \\ \oplus \\ R(-4)^2 \end{array} \longrightarrow 
     \begin{array}{c} R(-2) \\ \oplus \\ R(-3)^2  \\ \oplus \\ R(-6)^2 \end{array} 
 \longrightarrow A_{F_5} \longrightarrow 0
    $$
 \end{example}

The main idea behind the above example is the so-called doubling construction that we will recall now for sake of completeness.

\begin{definition}\label{doubling} Let $J\subset R$ be a homogeneous ideal of codimension $c$, such that $R/J$ is Cohen-Macaulay and $\omega_{R/J}$ is its canonical module. Furthermore, assume that $R/J$ satisfies the condition $G_0$ (i.e., it is Gorenstein at all minimal
primes). 
    Let $I$ be an ideal of codimension $c + 1$. $I$ is called a {\em doubling } of $J$ via $\psi$ if there exists  a short exact sequence of $R/J$ modules
\begin{equation}\label{eq:doubling}
0 \longrightarrow \omega_{R/J}(-d)\stackrel{\psi}{\longrightarrow} R/J \longrightarrow R/I\longrightarrow 0.
\end{equation}

By \cite[Proposition 3.3.18]{BH93}, if $I$ is a doubling, then $R/I$ is a Gorenstein ring. 
\end{definition}

Doubling plays an important role in the theory of Gorenstein liaison. Indeed, in \cite{KMMNP}, doubling is used to produce suitable Gorenstein divisors on arithmetically Cohen-Macaulay subschemes in several foundational constructions.
Doubling will also be crucial in this Section. Indeed,  the mapping cone of $\psi$ in \eqref{eq:doubling} gives a resolution of $R/I $. If it is minimal, then
one can read off the Betti table of $R/I$ from the Betti table of $R/J$. This mapping cone is the direct sum of the minimal free resolution $F_{\bullet }$ of $R/J$ with its dual (reversed) complex $\Hom(F_{\bullet},R)$ which justifies the terminology of "doubling".

\begin{Notation}
    For any linear form $\ell_i=a_i^1X_1+\cdots +a_i^nX_n\in [S]_1$ we will denote by $P_i\in \PP^{n-1}=\Proj(R)$ the point of homogeneous coordinates $P_i=(a_i^1:\cdots :a_i^n)$, i.e., the point $P_i$ corresponds to the dual of the linear form $\ell _i$. For any form $F_s=\ell_1^d+\cdots +\ell_s^d\in [S]_d$  where $\ell_1, \ldots ,\ell_s$ are general linear forms, we set $Z_{F_s}:=\{P_1,\ldots ,P_s\}\subset \PP^{n-1}$. We denote by $r_s$ the unique integer such that
  \vskip 1mm  $$ {n-2+r_s\choose n-1}<s\le {n-1+r_s\choose n-1}.$$
  \vskip 2mm
\end{Notation}

It is worthwhile to observe that under the hypothesis that $\ell_1, \ldots ,\ell_s$ are general linear forms, the Hilbert function of $R/I(Z_{F_s})$ is given by
 \vskip 1mm   
 $$
(1, \ n, \ {n+1\choose 2}, \ {n+2\choose 3}, \ \cdots  , \ \ {n+r_s-2\choose r_s-1}, \ s, \ s,  \cdots  ).
$$

 \vskip 1mm  
\begin{definition}  Fix integers $d,s\ge 1$ and $n\ge 2$ and let  $A_{F_s}$ be an  AG $\K$-algebra  of codimension $n$, socle degree $d$ and Macaulay dual generator $F_s=\ell_1^d+\cdots +\ell_s^d$ where $\ell_1, \ldots ,\ell_s$ are general linear forms. The 0-dimensional subscheme  $Z_{F_s}\subset \mathbb{P}^{n-1}$ is called a {\em tight annihilating scheme} of $A_{F_s}$.  
\end{definition}

\begin{definition}\label{compressed}
A \textit{compressed} AG $\K$-algebra of codimension $c$ and socle degree \(d\) is one for which the $h$-vector is as big as possible. So, its Hilbert function is
    $$
    \left( 1,\tbinom{c}{c-1},\dots,\tbinom{\frac{d}{2}+c-1}{c-1},\dots,\tbinom{c}{c-1},1 \right)
    $$
    if $d$ is even, and 
    $$
    \left( 1,\tbinom{c}{c-1},\dots,\tbinom{\frac{d-1}{2}+c-1}{c-1},\tbinom{\frac{d-1}{2}+c-1}{c-1},\dots,\tbinom{c}{c-1},1 \right)
    $$
    if $d$ is odd, thanks to the symmetry of the $h$-vector of any AG $\K$-algebra.    
\end{definition}

\begin{remark}\label{key2}
If $r_s>  \lfloor\frac{d}{2}\rfloor $, the $A_{F_s}$ is compressed and the MFR of $A_{F_s}$ is well known (see \cite{B} and \cite{MMN}).
So, from now on we will assume that  $r_s\le  \lfloor\frac{d}{2}\rfloor $. In this case, we have:
\begin{itemize}
    \item[(i)] $Z_{F_s}$ is the unique 0-dimensional scheme of length $s=S_{A_{F_s}}$ (the Sperner number of $A_{F_s}$);
    \item[(ii)] $I(Z_{F_s})\subset \ann_R(F_s)$;
    \item[(iii)] $
\HF_{R/I(Z_{F_s})}(t)=
\begin{cases}\binom{n+t-1}{t} \text{ for }  t\le r_s\\
S_{A_{F_s}} \text{ for } t> r_s;
\end{cases}
$
\item[(iv)] The quotient $\ann_R(F_s)/I(Z_{F_s})$ defines the dualizing sheaf of $R/I(Z_{F_s})$;
\item[(v)] $A_{F_s}$ is a doubling of $R/I(Z_{F_s})$;
\item[(vi)] The minimal free $R$-resolution of $A_{F_s}$
can be explicitly computed from the minimal free $R$-resolution of  $R/I(Z_{F_s})$.
\end{itemize}
\end{remark}

Therefore, the minimal free $R$-resolution  of a general set of points on a projective space will play a crucial role in determining a minimal free $R$-resolution of $A_{F_s}$ since the graded Betti numbers of the AG $\K$-algebra $A_{F_s}$ only depend on the graded Betti numbers of the tight annihilating scheme $Z_{F_s}$ of $A_{F_s}$. 

Let us recall what we know about  the minimal free resolution of a general set of points on a projective space. In \cite{L}, A. Lorenzini  predicted that the resolution is as simple as possible (having only two non-trivial rows for the Betti numbers). More precisely, we have:

\vskip 4mm
\noindent {\bf Minimal resolution conjecture (MRC):}  If $\Gamma \subset \PP^m$ is a general set of points, then for any integers $i, j,$ at most one of $\beta _{i,j}$ and $\beta_{i+1,j}$  is nonzero (i.e.  $\beta_{i,j}\cdot \beta_{i+1,j}=0$)   and the MFR of $I(\Gamma)$ has the following shape:

$$
0\longrightarrow  \begin{array}{c} R(-r_s-m)^{\beta_{m,r_s+m}} \\ \oplus \\ R(-r_s-m+1)^{\beta_{m,r_s+m-1}} \end{array}\longrightarrow \cdots
$$
\vskip 2mm
$$ \cdots
\longrightarrow \begin{array}{c} R(-r_s-2)^{\beta_{2,r_s+2}} \\  \oplus \\ R(-r_s-1)^{\beta_{2,r_s+1}} \end{array}
\longrightarrow \begin{array}{c} R(-r_s-1)^{\beta_{1,r_s+1}} \\ \oplus \\ R(-r_s)^{\beta_{1,r_s}} \end{array} \longrightarrow R \longrightarrow R/I(\Gamma) \longrightarrow 0.
$$

\vskip 4mm
The Minimal Resolution Conjecture has been proved or disproved in the following cases:

\begin{itemize}
    \item 
For any number of points in $\PP^2$ by Geramita and Maroscia \cite{GM};
\item For any number of points in $\PP^3$ by Ballico and Geramita \cite{BG};
\item For large numbers of points in any $\PP^r$ by Hirschowitz and Simpson \cite{HS};
\item For $r\ge 6$, $r\ne 9$, the MRC is false \cite{EPSW}.
\end{itemize} 
We are now ready to state the main result of this Section.

\begin{theorem}\label{mainthm1} Fix integers $d,s\ge 1$ and $n\ge 2$ and let $A_{F_s}$ be the AG $\K$-algebra  of codimension $n$, socle degree $d$ and Macaulay dual generator $F_s=\ell_1^d+\cdots +\ell_s^d$ where $\ell_1, \ldots ,\ell_s$ are general linear forms. 
Define  $r_s$ as the unique integer such that
    $ {n-2+r_s\choose n-1}<s\le {n-1+r_s\choose n-1}.$ If   $r_s\le  \lfloor\frac{d}{2}\rfloor $ then the MFR of $A_{F_s}$ has the following shape:
    $$ 
    0\longrightarrow  R(-d-n)  \longrightarrow  \begin{array}{c} F_{n-1} \\ \oplus \\ F_1^{\vee}(-d-n) \end{array}\longrightarrow \cdots $$
    $$
\longrightarrow \begin{array}{c} F_2 \\  \oplus \\ F_{n-2}^{\vee}(-d-n)  \end{array}
\longrightarrow \begin{array}{c} F_1 \\ \oplus \\ F_{n-1}^{\vee}(-d-n) \end{array} \longrightarrow R\longrightarrow A_{F_s} \longrightarrow 0
    $$
where $F_i$ is the i-th syzygy module of the ideal $I(Z_{F_s})$ of the tight annihilating scheme $Z_{F_s}$ of $A_{F_s}$ and $F_i^{\vee} = \operatorname{Hom}_R(F_i, R)$ denotes its algebraic dual module. 
    
\end{theorem}

\begin{proof}
We consider the minimal free $R$-resolution of $R/I(Z_{F_s})$   
\begin{eqnarray}\label{resR/IZ}
    0 \longrightarrow F_{n -1} \longrightarrow \cdots \longrightarrow F_{2} \longrightarrow F_{1} \longrightarrow R \longrightarrow R/I(Z_{F_s}) \longrightarrow 0. 
\end{eqnarray}
as well as the  minimal free $R$-resolution  of its canonical module $\text{Ext}_R^{n-1}(R/I(Z_{F_s}), R(-n-d)) \cong \omega_{R/I(Z_{F_s})} (-d)$:
\begin{eqnarray}\label{res2}
    0 \longrightarrow R(-d-n) \longrightarrow F_1^{\vee}(-d-n)
    \longrightarrow \cdots \end{eqnarray}
    $$\longrightarrow F_{n -1}^{\vee}(-d-n) \longrightarrow \omega_{R/I(Z_{F_s})} (-d)\longrightarrow 0.
$$

By hypothesis  $r_s \le \lfloor d/2 \rfloor$. Therefore, using Definition \ref{doubling} and Remark \ref{key2}, we get that $Z_{F_s}$ is the tight annihilating scheme of $A_{F_s}$ and the AG $\K$-algebra $A_{F_s}$ is a doubling of $R/I(Z_{F_s})$. Thus, there exists a short exact sequence of $R$-modules
\begin{eqnarray*}
    0 \longrightarrow \omega_{R/I(Z_{F_s})}(-d) \xrightarrow{\psi} R/I(Z_{F_s}) \longrightarrow A_{F_s} \longrightarrow 0.
\end{eqnarray*}

Finally,  the mapping cone of the morphism $\psi$ yields a free $R$-resolution for $A_{F_s}$. The modules of this resolution are given by the direct sum of the resolution (\ref{resR/IZ}) and its dual  (\ref{res2}):

\begin{eqnarray}\label{mfr}
        0\longrightarrow  R(-d-n)  \longrightarrow  \begin{array}{c} F_{n-1} \\ \oplus \\ F_1^{\vee}(-d-n) \end{array}\longrightarrow \cdots
        \end{eqnarray}
        $$
        \longrightarrow \begin{array}{c} F_2 \\  \oplus \\ F_{n-2}^{\vee}(-d-n)  \end{array}
\longrightarrow \begin{array}{c} F_1 \\ \oplus \\ F_{n-1}^{\vee}(-d-n) \end{array} \longrightarrow R\longrightarrow A_{F_s} \longrightarrow 0
$$

  For the minimality of this last resolution, we analyze the degrees of the generators. For a set of $s$ points in general position, the generators of $F_i$ are concentrated in degrees $i+r_s-1$ and $i+r_s$. In particular, the maximum degree in $F_i$ is $i+r_s$, while the minimum degree in $F_{n-i}^{\vee}(-d-n)$ is $d+n-r_s-(n-i)=d-r_s+i$. Since by hypothesis $r_s \le \lfloor d/2 \rfloor$, it follows that $i+r_s < d-r_s+i+1$ for all $i$ and  so we conclude that resolution  (\ref{mfr}) is minimal.
\end{proof}

\section{Jordan type}\label{section5}
The goal of this Section is to compute the generic linear Jordan type of $A_{F_s}$.

\begin{theorem}\label{mainthm12} We fix the integers $d \geq 1$, $ 1 \le s \le \binom{n+d-1}{d}$ and $n\ge 2$. We consider the AG $\K$-algebra $A_{F_s}$ of codimension $n$, socle degree $d$ and Macaulay dual generator $F_s=\ell_1^d+\cdots +\ell_s^d$ where $\ell_1, \ldots ,\ell_s$ are general linear forms.   The generic linear Jordan type of $A_{F_s}$ is:
$$ P_{\ell,A_{F_s}}= ((d + 1), (d-1)^{\alpha_1},(d - 3)^{\alpha_2},\cdots,  (d+1-2i)^{\alpha _i}, \cdots , (d+1-2r_s)^{\alpha_{r_s}})$$
where $\alpha _i=h_i-h_{i-1}$, $0\le i \le r_s$ and $(h_0, \dots, h_d)$ is the $h$-vector of $A_{F_s}$.
\end{theorem}
\begin{proof}
        
Let $\ell $ be a generic linear form as well as a weak Lefschetz element for $A_{F_s}$ which exists because, by Theorem  \ref{WLP_main}, $A_{F_s}$ satisfies the WLP. 
We will compute $P_{\ell,A_{F_s}}$ using induction on $d$. If $d=1$ or $2$ the \ result is obvious. Indeed, if $d=1$ then, after a change of coordinates, we have $F_s=x_1+\cdots +x_s$ with  $s\le n$, $ {F_s}$ has  $h$-vector (1, 1) and $P_{\ell, A_{F_s}}=(2)$. If $d=2$ then $F_s=\ell_1^2+\cdots +\ell_s^2$ with  $1\le s\le {n+1\choose 2}$,
$A_{F_s}$ has $h$-vector:  $ \ (1, \ \min(s,n), \ 1)$, and
$P_{\ell, A_{F_s}}=(3, 1^{\alpha_1})$ with $\alpha _1=\min(s,n)-1$. 

Assume $d>2$. We distinguish 2 cases:

\vskip 2mm
\noindent {\bf Case 1: $r_s<\frac{d+1}{2}$.} We consider the AG $\K$-algebra $B=A_{\ell' \circ F_s}$ where $\ell'=x_1+\cdots +x_n$. 
By construction $B$ is an AG $\K$-algebra of codimension $n$, socle degree $d-1$ and Macaulay dual generator $\ell '\circ F_s=L_1^{d-1}+\cdots +L_s^{d-1}$  where $L_1, \ldots ,L_s$ are general linear forms in $\K[X_1,\ldots,X_n]$. We also observe that $\HF_B(t)=\HF_{A_{F_s}}(t)$ for all $t\le r_s$. By hypothesis of induction
$$ P_{\ell,B}= ((d), (d-2)^{\alpha_1},(d - 4)^{\alpha_2},\cdots,  (d-2i)^{\alpha _i}, \cdots , (d-2r_s)^{\alpha_{r_s}}).$$

Since $A_{F_s}$ and $B$ satisfy the WLP and $S_B=S_{A_{F_s}}$, we have that the numbers of parts in $P_{\ell, B}$
coincide with the numbers of parts in $P_{\ell, A_{F_s}}$. Moreover,
$$
\dim _{\K}A_{F_s}=\dim_{\K}B+S_B$$
and the number of beads in each part of $P_{\ell, A_{F_s}}$ is the result of adding one more to the number of beads in each part of $P_{\ell, B}$. Therefore, we conclude that
$$ P_{\ell,A_{F_s}}= ((d+1), (d-1)^{\alpha_1},(d - 3)^{\alpha_2},\cdots,  (d+1-2i)^{\alpha _i}, \cdots , (d+1-2r_s)^{\alpha_{r_s}})$$
as we claim.
\vskip 2mm
\noindent {\bf Case 2: $r_s=\frac{d}{2}$.} (So, necessarily $d$ is even.)  In this case, we also consider the AG $\K$-algebra $B=A_{\ell' \circ F_s}$ where $\ell'=x_1+\cdots +x_n$. 
By construction $B$ is an AG $\K$-algebra of codimension $n$, socle degree $d-1$ and Macaulay dual generator $\ell '\circ F_s=L_1^{d-1}+\cdots +L_s^{d-1}$  where $L_1, \ldots ,L_s$ are general linear forms in $\K[X_1,\ldots,X_n]$. In this case $B$ is a compressed AG $\K$-algebra with socle degree $d-1$ and generic linear Jordan type
$$ P_{\ell,B}= ((d), (d-2)^{\alpha_1},(d - 4)^{\alpha_2},\cdots,  (d-2i)^{\alpha _i}, \cdots , (d-2(r_s-1))^{\alpha_{r_s-1}}).$$

Since $A_{F_s}$ and $B$ satisfy the WLP and $S_{A_{F_s}}= S_B+ h_{d/2}-h_{d/2 -1}$, we have that the numbers of parts in $P_{\ell, A_{F_s}}$
coincide with that the numbers of parts in $P_{\ell, B}$ plus $h_{d/2}-h_{d/2 -1}$. In addition,
$$
\dim _{\K}A_{F_s}=\dim_{\K}B+S_B+  h_{d/2}-h_{d/2 -1}.$$
Therefore, the only possibility is to add a bead in each part of $ P_{\ell,B}$ and $h_{d/2}-h_{d/2 -1}$ parts with only one bead (notice that $d+1-2r_s=1)$, i.e.,
$$ P_{\ell,A_{F_s}}= ((d+1), (d-1)^{\alpha_1},(d - 3)^{\alpha_2},\cdots,  (d+1-2i)^{\alpha _i}, \cdots , (d+1-2(r_s-1))^{\alpha_{r_s-1}},(d+1-2r_s)^{\alpha_{r_s}})$$
which proves what we want.
\end{proof}

To determine whether an Artinian standard graded $\K$-algebra $A$ has the WLP seems a simple problem of linear algebra, but it has proven to be extremely elusive and much more work on this topic remains to be done, see \cite{JMR} and its reference list for more information. 
It is an interesting fact that the Jordan type of a linear form $\ell$ in an Artinian $\K$-algebra $A$ determines whether or not $\ell$ is a weak/strong Lefschetz element of $A$, as the following result shows.

\begin{definition}
    Let $A$ be a standard graded Artinian $\K$-algebra with Hilbert function $\HF_A$. We consider the partition whose parts are the non-zero values of $\HF_A$, after reordering them to become non-increasing. We call the conjugate of this partition the \emph{conjugate partition of $\HF_A$}.
\end{definition}

\begin{proposition}\label{jordanTypeWLP}
    Let $A$ be a standard graded Artinian $\K$-algebra. Then a linear form $\ell \in [A]_1$ is a weak Lefschetz element if and only if the number of parts in the Jordan type $P_{\ell,A}$ equals the Sperner number $S_A$.
   $A$  satisfies the SLP if and only if the conjugate partition
of $\HF_A$ equals the Jordan type $P_{\ell,A}$ for some
linear form $\ell \in [A]_1$.
\end{proposition}
\begin{proof}
The reader can look at  \cite[Remark 3 and Proposition 14]{HW} or \cite[Proposition 3.64]{HMMNWW}.   
\end{proof}

Therefore, as an immediate application of Theorem \ref{mainthm12} we get the following result

\begin{theorem} \label{SLP_main}
We consider the AG $\K$-algebra $A_{F_{s}} = R/\operatorname{Ann}(F_{s}),$ of codimension $n$, socle degree $d$ and Macaulay dual generator 
\begin{eqnarray*}
F_{s}=\ell_1^d+\cdots+\ell_s^d
\end{eqnarray*}
where $\ell_1,\dots,\ell_s$ are general linear forms in $[S]_1$, and $1 \le s \le \binom{d+n-1}{d}$.                  
Then, $A_{F_s}$ has the SLP. 
\end{theorem}
\begin{proof} By Theorem \ref{mainthm12}  the generic linear Jordan type of $A_{F_s}$ is:
$$ P_{\ell,A_{F_s}}= ((d + 1), (d-1)^{\alpha_1},(d - 3)^{\alpha_2},\cdots,  (d+1-2i)^{\alpha _i}, \cdots , (d+1-2r_s)^{\alpha_{r_s}})$$
where $\alpha_i = h_i - h_{i-1}$ for $1 \le i \le r_s$, and $(h_0, \dots, h_d)$ denotes the $h$-vector of $A_{F_s}$.
Since the Hilbert function of $A_{F_s}$ is given by (see Proposition \ref{computingHilbertFunctionGorenstein})
$$
h_k=\dim_{\mathbb K}[A_{F_{s}}]_k
=\begin{cases}\min\left\{\binom{n+k-1}{k},\, s\right\} \text{ for }  0\leq k\leq\lfloor\frac{d}{2}\rfloor \\
\text{symmetry}  \text{ for } \lfloor\frac{d}{2}\rfloor < k \leq d  ,
\end{cases}
$$
we get that the conjugate partition
of $\HF_{A_{F_s}}$ equals the Jordan type $P_{\ell,A_{F_s}}$.
 Therefore,  applying Proposition \ref{jordanTypeWLP} we conclude that
$A_{F_s}$ has the SLP.
\end{proof}

\begin{remark} \label{WLPnecessity} 
We note that while the SLP (Theorem \ref{SLP_main}) implies the WLP (Theorem \ref{WLP_main}), establishing the WLP first was essential to our approach.
Indeed, the existence of a weak Lefschetz element $\ell$ is crucial to ensure that the number of parts in the partition of the generic linear Jordan type $P_{\ell, A_{F_s}}$ from \ref{mainthm12} matches or behaves controllably with respect to the algebra $B$, allowing the full determination of the Jordan blocks.
\end{remark}

\begin{example} We consider the AG $\K$-algebra $A_{F_8}$ of codimension $4$, socle degree $7$ and Macaulay dual generator $F_8=X^7+Y^7+Z^7+T^7+\ell_1^7+\ell_2^7+\ell_3^7+\ell_4^7$ where $\ell_1,\ldots ,\ell_4$ are general linear forms. According to Proposition \ref{h_vector}
 it has $h$-vector (1,4,8,8,8,8,4,1).  Using Macaulay2  \cite{M2}, we compute the Jordan type of $x+y+z+t$ on $A_{F_s}$ and we get: 
$$
P_{x+y+z+t,A_{F_s}} = (8,6,6,6,4,4,4,4),
$$
since there are strings of  lengths $8$, $6$, $6$, $6$, $4$, $4$, $4$ and $4$, respectively. In addition, the strings start at degrees $0,1,1,1,2,2,2,2$, so the Jordan degree-type is
$$ 
\mathcal{S}_{x+y+z+t,A_{F_s}} = ((8,0),(6,1),(6,1),(6,1),(4,2),(4,2),(4,2),(4,2)) = (8_0, 6_1^3, 4_2^4) . 
$$

\end{example}

\section*{Declaration}

\subsection*{Funding} The first author was supported by CAPES/PDSE grand number 88881.220483/2025-01. The second author was partially supported
by the grant PID2024-157142NB-I00.
\subsection*{Conflict of interest} On behalf of all authors, the corresponding author states that there is no conflict of
interest.


\begin{thebibliography}{99}

\bibitem{BG}  E. Ballico, A.V. Geramita, {\em The minimal free resolution of s general points in $\PP^3$}, Can. Math. Soc.
Conf. Proc. {\bf 6} (1986), l--10.

\bibitem{B} M. Boij, {\em Betti numbers of compressed level algebras },     Journal of Pure and Applied Algebra {\bf 134} (1999), 111--131.               

\bibitem{BK}H. Brenner and A. Kaid, {\em Syzygy bundles on $\mathbb{P}^2$ and the weak Lefschetz property}, Illinois J. Math. {\bf 51} (2007), no.~4, 1299--1308.

\bibitem{BH93}
W.~Bruns and J.~Herzog, \emph{Cohen-{M}acaulay rings}, Cambridge Studies in Advanced Mathematics, vol.~39, Cambridge University Press (1993), Cambridge.


\bibitem{EPSW} D. Eisenbud, S. Popescu, F. Schreyer and, C. Walter, {\em Exterior algebra methods for the minimal resolution conjecture}, Duke Math. J. {\bf 112} (2002), no.~2, 379--395.

\bibitem{G1} F. Gaeta, {\em Sur la distribution des degr\'es des formes appartenant \`a{} la matrice de l'id\'eal homog\`ene attach\'e{} \`a{} un groupe de $N$ points g\'en\'eriques du plan}, C. R. Acad. Sci. Paris {\bf 233} (1951), 912--913.


\bibitem{GM} A.V. Geramita, P. Maroscia, {\em The ideal of forms vanishing at a finite set of points $\PP^n$}. J. Algebra
{\bf 90} (1984), 528-555.

\bibitem{M2}
D.~R.~Grayson and M.~E.~Stillman,
\emph{Macaulay2, a software system for research in algebraic geometry}.
Available at \url{https://macaulay2.com/}.

\bibitem{HMMNWW} T. Harima, T. Maeno, H. Morita, Y. Numata, A. Wachi, and J. Watanabe.
{\em The Lefschetz Properties}. Vol. 2080. Lecture Notes in Mathematic. Springer,
Heidelberg, 2013, pp. xx+250.
\bibitem{HW} T. Harima and J. Watanabe. {\em The finite free extension of Artinian K-algebras
with the strong Lefschetz property}.  Rendiconti del Seminario Matematico
della Università di Padova {\bf 110} (2003),  119--146.

\bibitem{HS} A. Hirschowitz and C.~T. Simpson, {\em La r\'esolution minimale de l'id\'eal d'un arrangement g\'en\'eral d'un grand nombre de points dans $\PP^n$}, Invent. Math. {\bf 126} (1996), no.~3, 467--503.

\bibitem{IMM22} A. Iarrobino, P. Macias Marques, and C. McDaniel. {\em Artinian algebras and
Jordan type}. In: Journal of Commutative Algebra 14.3 (2022), pp. 365--414.

\bibitem{IK}A. Iarrobino and V. Kanev, {\it Power sums, Gorenstein algebras, and determinantal loci}, Lecture Notes in Mathematics, 1721, Springer, Berlin, 1999; MR1735271


\bibitem{JMR}M. Juhnke-Kubitzke and R. M. Miró-Roig. {\em List of problems. } In: Nagel, U., Adiprasito,
K., Di Gennaro, R., Faridi, S., Murai, S. (eds) Lefschetz Properties. SLP-WLP 2022. Springer INdAM
Series, vol 59. Springer, Singapore. https://doi.org/10.1007/978-981-97-3886-1-12.

\bibitem{KMMNP} 
J. Kleppe, J.  Migliore, R.~M.  Mir\' o-Roig, U. Nagel and  C. Peterson, {\em Gorenstein liaison, complete intersection
liaison invariants and unobstructedness}, Mem.\ Amer.\ Math.\
  Soc.\ {\bf 154} (2001), no.\ 732 viii+116 pp.



 \bibitem{L} A. M. Lorenzini, {\em The minimal resolution conjecture}, J. Algebra {\bf 156} (1993), no. 1, 5--35.

  
\bibitem{MW09}
T. Maeno and J. Watanabe, {\em Lefschetz elements of Artinian Gorenstein algebras and Hessians of homogeneous polynomials}, Illinois J. Math. {\bf 53} (2009), no.~2, 591--603. 



\bibitem{MMRMO}
E. Mezzetti, R.~M. Mir\'o-Roig and G.~M. Ottaviani, {\em Laplace equations and the weak Lefschetz property}, Canad. J. Math. {\bf 65} (2013), no.~3, 634--654. 

\bibitem{MMN}  J.  Migliore, R.~M.  Mir\' o-Roig and U. Nagel, {\em Minimal resolution of relatively compressed level algebras}, Journal of Algebra
{\bf 284} (2005),  333-370.

\bibitem{MMN2} J.~C. Migliore, R.~M. Mir\'o-Roig and U. Nagel, {\em Monomial ideals, almost complete intersections and the weak Lefschetz property}, Trans. Amer. Math. Soc. {\bf 363} (2011), no.~1, 229--257; MR2719680




\bibitem{S80}R.~P. Stanley, {\em Weyl groups, the hard Lefschetz theorem, and the Sperner property}, SIAM J. Algebraic Discrete Methods {\bf 1} (1980), no.~2, 168--184.

\end{thebibliography}
\end{document}